\documentclass[reqno,11pt]{amsart}
\usepackage[margin=1in]{geometry}

\usepackage{amsthm, amsmath, amssymb, bm, mathtools}
\usepackage{tikz}

\usepackage[utf8]{inputenc}
\usepackage[T1]{fontenc}

\usepackage[textsize=small,backgroundcolor=orange!20]{todonotes}

\usepackage[hidelinks]{hyperref}
\usepackage{url}

\usepackage[noabbrev,capitalize]{cleveref}
\crefname{equation}{}{}

\usepackage{url}
\usepackage[color,notcite,final]{showkeys} 

\colorlet{refkey}{orange!20}
\colorlet{labelkey}{blue!60}

\newtheorem{theorem}{Theorem}[section]

\newtheorem{lemma}[theorem]{Lemma}

\newtheorem{conjecture}[theorem]{Conjecture}

\theoremstyle{definition}

\theoremstyle{remark}

\newcommand{\EE}{\mathbb{E}}

\newcommand{\RR}{\mathbb{R}}

\newcommand{\ZZ}{\mathbb{Z}}

\title{More on lines in Euclidean Ramsey theory}

\author{David Conlon}
\address{Department of Mathematics, California Institute of Technology, Pasadena, CA 91125, USA}
\email{dconlon@caltech.edu}
\thanks{Research supported by NSF Award DMS-2054452.}
\author{Yu-Han Wu}
\address{
\'Ecole Normale Sup\'erieure - PSL, Paris, France}
\email{yu-han.wu@ens.psl.eu}

\begin{document}

\begin{abstract}
Let $\ell_m$ be a sequence of $m$ points on a line with consecutive points at distance one. Answering a question raised by Fox and the first author and independently by Arman and Tsaturian, we show that there is a natural number $m$ and a red/blue-colouring of $\mathbb{E}^n$ for every $n$ that contains no red copy of $\ell_3$ and no blue copy of $\ell_m$.
\end{abstract}

\maketitle

\section{Introduction}

Let $\EE^n$ denote $n$-dimensional Euclidean space, that is, $\RR^n$ equipped with the Euclidean metric. Given two sets $X_1, X_2 \subset \EE^n$, we write $\mathbb{E}^n \rightarrow (X_1, X_2)$ if every red/blue-coloring of $\EE^n$
contains either a red copy of $X_1$ or a blue copy of $X_2$, where a copy for us will always mean an isometric copy. Conversely, $\mathbb{E}^n \nrightarrow (X_1, X_2)$ means that there is some red/blue-coloring of $\EE^n$ which contains neither a red copy of $X_1$ nor a blue copy of $X_2$. 

The study of which sets $X_1, X_2 \subset \EE^n$ satisfy $\mathbb{E}^n \rightarrow (X_1, X_2)$ is a particular case of the Euclidean Ramsey problem, which has a long history going back to a series of seminal papers~\cite{EGMRSS1, EGMRSS2, EGMRSS3} of Erd\H{o}s, Graham, Montgomery, Rothschild, Spencer and Straus in the 1970s. Despite the vintage of the problem, surprisingly little progress has been made since these foundational papers (though see~\cite{FR, K91} for some important positive results). For instance, it is an open problem, going back to the papers  of Erd\H{o}s et al.~\cite{EGMRSS2}, as to whether, for every $n$, there is $m$ such that $\mathbb{E}^n \nrightarrow (X, X)$ for every $X \subset \EE^n$ with $|X| = m$. 

Write $\ell_m$ for the set consisting of $m$ points on a line with consecutive points at distance one. Perhaps because it is a little more accessible than the general problem, the question of determining which $n$ and $X$ satisfy the relation $\EE^n \rightarrow (\ell_2, X)$ has received considerable attention. For instance, it is known~\cite{Juh, Tsat} that $\mathbb{E}^2 \rightarrow (\ell_2,X)$ for every four-point set $X \subset \mathbb{E}^2$ and that $\mathbb{E}^2 \rightarrow (\ell_2,\ell_5)$. On the other hand~\cite{CsTo}, there is a set $X$ of $8$ points in the plane, namely, a regular heptagon with its center, such that $\mathbb{E}^2 \nrightarrow (\ell_2,X)$. 

In higher dimensions, by combining results of Szlam~\cite{Sz01} and Frankl and Wilson~\cite{FW}, it was observed by Fox and the first author~\cite{CF19} that $\mathbb{E}^n \rightarrow (\ell_2,\ell_m)$ provided $m \leq 2^{cn}$ for some positive constant $c$ (see also~\cite{ArmTsat, ArmTsat2} for some better bounds in low dimensions). 
Our concern here will be with a question raised independently by Fox and the first author~\cite{CF19} and also by Arman and Tsaturian~\cite{ArmTsat2}, namely, as to whether an analogous result holds with $\ell_2$ replaced by $\ell_3$. That is, for every natural number $m$, is there a natural number $n$ such that $\EE^n \rightarrow (\ell_3, \ell_m)$? We answer this question in the negative.

\begin{theorem} \label{thm:main1}
There exists a natural number $m$ such that $\mathbb{E}^n \nrightarrow (\ell_3, \ell_m)$ for all $n$.
\end{theorem}

Before our work, the best result that was known in this direction was a 50-year-old result of Erd\H{o}s et al.~\cite{EGMRSS1}, who showed that $\mathbb{E}^n \nrightarrow (\ell_6, \ell_6)$ for all $n$. Their proof uses a spherical colouring, where all points at the same distance from the origin receive the same colour. We will also use a spherical colouring, though, unlike the colouring in~\cite{EGMRSS1}, which is entirely explicit, our colouring will be partly random.

\section{Preliminaries}

In this short section, we note two key lemmas that will be needed in our proof. The first says that certain real-valued quadratic polynomials are reasonably well-distributed modulo a prime $q$.

\begin{lemma} \label{lem:dist}
Let $p(x) = x^2 + \alpha x + \beta$,  where $\alpha$ and $\beta$ are real numbers, and let $q$ be a prime number. Then, for $m = q^3$, the set $\{p(i)\}_{i=1}^m$ overlaps with at least $q/6$ of the intervals $[j, j+1)$ with $0 \leq j \leq q -1$ when considered mod $q$.
\end{lemma}

\begin{proof}
By a standard argument using the pigeonhole principle, there exists some $k \leq q^2$ such that $|k \alpha| \leq 1/q$ mod $q$. We split into two cases, depending on whether $k$ is a multiple of $q$ or not. 

Suppose first that $k \not\equiv 0 \bmod q$ and consider the set of values $\{p(ki)\}_{i=1}^q$. Note first that $\{i^2\}_{i=1}^q$ is a set of $(q+1)/2$ distinct integers mod $q$, so, since $k$ is not a multiple of $q$, the same is also true of the set $\{k^2 i^2\}_{i=1}^q$. Hence, letting $p_1(x) = x^2 + \beta$, we see that the set $\{p_1(ki)\}_{i=1}^q$ overlaps with at least $q/2$ of the intervals $[j, j+1)$ with $0 \leq j \leq q -1$ when considered mod $q$. But $|ki\alpha| \leq 1$ mod $q$ for all $1 \leq i \leq q$, so that $|p(ki) - p_1(ki)| \leq 1$ for all $1 \leq i \leq q$. Therefore, since exactly three different intervals are within distance one of any particular interval, the set  $\{p(ki)\}_{i=1}^q$ overlaps with at least $q/6$ of the intervals $[j, j+1)$ mod $q$.

Suppose now that $k = sq$ for some $s \leq q$. Then $sq\alpha = rq + \epsilon$ for some $|\epsilon| \leq 1/q$, which implies that $\alpha = \frac{r}{s} + \epsilon'$, where $|\epsilon'| \leq 1/q^2$. Without loss of generality, we may assume that $r$ and $s$ have no common factors. Consider now the polynomial $p_2(x) = x^2 + \frac{r}{s} x$ and the set $\{p_2(si)\}_{i=1}^q$. Since $p_2(si) = s^2 i^2 + r i$, it is easy to check that $p_2(si) \equiv p_2(sj) \bmod  q$ if and only if $s^2(i+j) + r \equiv 0 \bmod q$. Since $r$ and $s$ are coprime, this implies that the set $\{p_2(si)\}_{i=1}^q$ takes at least $q/2$ values mod $q$. 
Hence, letting $p_3(x) = x^2 + \frac{r}{s} x + \beta$, we see that the set $\{p_3(si)\}_{i=1}^q$ overlaps with at least $q/2$ of the intervals $[j, j+1)$ with $0 \leq j \leq q -1$ when considered mod $q$. But, since $|\alpha - r/s| \leq 1/q^2$, we have that
$|p(si) - p_3(si)| = |\alpha - \frac{r}{s}| si \leq 1$, so that, as above, the set  $\{p(si)\}_{i=1}^q$ overlaps with at least $q/6$ of the intervals $[j, j+1)$ mod $q$.
\end{proof}

Given $M$ real polynomials $p_1, \dots, p_M$ in $N$ variables, 
a vector $\sigma \in \{-1, 0, 1\}^M$ is called a sign pattern of  $p_1, \dots, p_M$ if there exists some $x \in  \mathbb{R}^N$ such that the sign of $p_i(x)$ is $\sigma_i$ for all $1 \leq i \leq M$. The second result we need is the Oleinik--Petrovsky--Thom--Milnor theorem (see, for example,~\cite{BPR06}), which, for $N$ fixed, gives a polynomial bound for the number of sign patterns.

\begin{lemma} \label{lem:MT}
For $M \geq N \geq 2$, the number of sign patterns of $M$ real polynomials in $N$ variables, each of degree at most $D$, is at most $\left(\frac{50 DM}{N}\right)^N$.
\end{lemma}

\section{Proof of Theorem~\ref{thm:main1}}

Suppose that $a_1, a_2, a_3 \in \mathbb{R}^n$ form a copy of $\ell_3$ with $|a_1 - a_2| = |a_2 - a_3| = 1$. If the points are at distances $x_1$, $x_2$ and $x_3$, respectively, from the origin $o$ and the angle $a_1a_2o$ is $\theta$, then we have
\[x_1^2 = x_2^2 + 1 -2x_2 \cos \theta\]
and
\[x_3^2 = x_2^2 + 1 + 2x_2 \cos \theta.\]
Adding the two gives
\[x_1^2 + x_3^2 = 2x_2^2 + 2.\]
Similarly, if $a_1, a_2, \dots, a_m \in  \mathbb{R}^n$ form a copy of $\ell_m$ with $|a_i - a_{i+1}| = 1$ for all $i = 1, 2, \dots, m-1$, then, again writing $x_i$ for the distance of $a_i$ from the origin, we have 
\[x_{i-1}^2 + x_{i+1}^2 = 2 x_i^2 + 2\]
for all $i = 2, \dots, m-1$. Given these observations, our aim will be to colour $\mathbb{R}_{\geq 0}$ so that there is no red solution to $y_1 + y_3 = 2 y_2 + 2$ and no blue solution to the system $y_{i-1} + y_{i+1} = 2 y_i + 2$ with $i = 2, \dots, m-1$. Assuming that we have such a colouring $\chi$, we can simply colour a point $a \in \mathbb{R}^n$ by $\chi(|a|^2)$ and it is easy to check that there is no red copy of $\ell_3$ and no blue copy of $\ell_m$.

We have therefore moved our problem to one of finding a natural number $m$ and a colouring $\chi$ of $\RR_{\geq 0}$ with no red solution to $y_1 + y_3 = 2 y_2 + 2$ and no blue solution to the system $y_{i-1} + y_{i+1} = 2 y_i + 2$ with $i = 2, \dots, m-1$. Let $q$ be a prime number. We will take $m = q^3$ and define $\chi$ by choosing an appropriate colouring $\chi'$ of $\ZZ_q$ and then setting $\chi(y) = \chi'(\lfloor y \rfloor \bmod q)$ for all $y \in \RR_{\geq 0}$. Our aim now is to show that there is a suitable choice for $\chi'$. For this, we consider a random red/blue-colouring $\chi'$ of $\ZZ_q$ and show that, for $q$ sufficiently large, the probability that $\chi$ contains either of the banned configurations is small.

Concretely, suppose that $\ZZ_q$ is coloured randomly in red and blue with each element of $\ZZ_q$ coloured red with probability $p = q^{-3/4}$ and blue with probability $1-p$. With this choice, the expected number of solutions in red to any of the equations $y_1 + y_3 = 2 y_2 + c$ with $c \in \{1, 2, 3\}$ is at most
\[3p^3 q^2 + 9 p^2 q < 12 q^{-1/4} < \frac 12,\]
where we used that there are at most $3q$ solutions to any of our $3$ equations with two of the variables $\{y_1, y_2, y_3\}$ being equal and that $q$ is sufficiently large. Note that if there are indeed no red solutions to these three equations over $\ZZ_q$, then there is no red solution to $y_1 + y_3 = 2 y_2 + 2$ in the colouring $\chi$ of $\RR$. Indeed, if $y_i = n_i + \epsilon_i$ with $0 \leq \epsilon_i < 1$, then $n_i$ is coloured red in $\chi'$ and 
\[n_1 + n_3 = 2 n_2 + 2 + 2\epsilon_2 - \epsilon_1 - \epsilon_3.\]
But $|2\epsilon_2 - \epsilon_1 - \epsilon_3| < 2$, so we must have
\[n_1 + n_3 = 2 n_2 + c\]
for $c \in \{1, 2, 3\}$. However, we know that there are no red solutions to any of these equations in the colouring $\chi'$, so there is no red solution to $y_1 + y_3 = 2 y_2 + 2$ in the colouring $\chi$.

For the blue configurations, we first observe that if the $y_i$ satisfy the equations $y_{i-1} + y_{i+1} = 2 y_i + 2$ with $i = 2, \dots, m-1$ with $y_1 = a$ and $y_2 = a + d$, then $y_i = a + (i-1)d + (i^2 - 3i + 2)$. In particular, by Lemma~\ref{lem:dist}, at least $q/6$ elements of the sequence $y_1, \dots, y_m$ lie in different intervals $[j, j+1)$ with $0 \leq j \leq q-1$ when considered mod $q$. 

Our aim now is to apply Lemma~\ref{lem:MT} to count the number of different ways in which a set of solutions $(y_1, y_2, \dots, y_m)$ to our system of equations can overlap the collection of intervals $[j, j+1)$ mod $q$. Without loss of generality, we may assume that $0 \leq a, d < q$. Since, under this assumption, any set of solutions over $\mathbb{R}$ to our system of equations is contained in the interval $[0,2m^2)$, it will suffice to count the number of feasible overlaps with the intervals $[j, j+1)$ with $0 \leq j \leq 2m^2-1$. Since we need to check at most two linear inequalities
in the two variables $a$ and $d$ to check whether each of the $m$ points are placed in each of the $2m^2$ intervals, we can apply Lemma~\ref{lem:MT} with $N = 2$, $D = 1$ and $M = 2 \cdot m \cdot 2m^2 =4m^3$ to conclude that the points $y_1, \dots, y_m$ overlap the intervals $[j,j+1)$ with $0 \leq j \leq 2m^2-1$ in at most $(100m^3)^{2} = 10^4 m^6$ different ways. 
But now, since at least $q/6$ of the $y_i$ must always be in distinct intervals, a union bound implies that the probability we have a blue solution to our system of equations is at most
\[10^4 m^6 (1 - q^{-3/4})^{q/6} <  \frac 12\]
for $m$ sufficiently large. Combined with our earlier estimate for the probability of a red solution to $y_1 + y_3 = 2 y_2 + 2$, we see that for $m$ sufficiently large ($m = 10^{50}$ will suffice) there exists a colouring with no red $\ell_3$ and no blue $\ell_m$, as required.

\section{Concluding  remarks}

We say that a set $X \subset \mathbb{E}^d$ is {\it Ramsey} if for every natural number $r$ there exists $n$ such that every $r$-colouring of $\mathbb{E}^n$ contains a monochromatic copy of $X$. In~\cite{CF19}, it was shown that 
a set $X$ is Ramsey if and only if for every natural number $m$ and every fixed $K \subset \mathbb{E}^m$ there exists $n$ such that $\EE^n \rightarrow (X, K)$. We suspect that there may be an even simpler characterisation.

\begin{conjecture} \label{conj:char}
A set $X$ is Ramsey if and only if for every natural number $m$ there exists $n$ such that $\EE^n \rightarrow (X, \ell_m)$.
\end{conjecture}

Of course, by the result mentioned  above, we already know that if $X$ is Ramsey, then $\EE^n \rightarrow (X, \ell_m)$ for $n$ sufficiently large. It therefore remains to show that if $X$ is not Ramsey, then there exists $m$ such that $\EE^n \nrightarrow (X, \ell_m)$ for all $n$. To prove this in full generality might be difficult. However, an important result of Erd\H{o}s et al.~\cite{EGMRSS1} says that if $X$ is Ramsey, then it must be spherical, in the sense that it must be contained in the  surface of a sphere of some dimension. Thus, a first step towards Conjecture~\ref{conj:char} might be to prove the following.

\begin{conjecture} \label{conj:spher}
For every non-spherical set $X$, there exists a natural number $m$ such that $\EE^n \nrightarrow (X, \ell_m)$ for all $n$. 
\end{conjecture}

The simplest example of a non-spherical set is the line $\ell_3$, so our main result may be seen as a verification of Conjecture~\ref{conj:spher} in this particular case. The next case of interest seems to be when $X$ consists of three points $a_1, a_2, a_3$ on a line, but now with $|a_1 - a_2|  =  1$ and $|a_2 - a_3| = \alpha$ for some irrational $\alpha$.

\end{document}